\documentclass[12pt]{amsproc}
\usepackage{amsmath,amsthm}
\usepackage{amsfonts,amssymb}
\newtheorem{thm}{Theorem}[section]

\newtheorem{lem}[thm]{Lemma}
\newtheorem{prop}[thm]{Proposition}

\theoremstyle{remark}

\newcommand{\be}{\begin{equation}}
\newcommand{\ee}{\end{equation}}
\newcommand{\bea}{\begin{eqnarray}}

\newcommand{\eea}{\end{eqnarray}}
\newcommand{\Bea}{\begin{eqnarray*}}
\newcommand{\Eea}{\end{eqnarray*}}

\def\CF{{\mathcal F}}

\def\CH{{\mathcal H}}

\def\C{{\mathbb C}}
\def\H{{\mathbb H}}
\def\N{{\mathbb N}}
\def\R{{\mathbb R}}

\def\1{\text{\bf {1}}}

\begin{document}

\title[Riesz transforms and multipliers]
{ Riesz transforms and multipliers for the Grushin operator}
\author{ K. Jotsaroop, P. K. Sanjay  and S. Thangavelu}

\address{Department of Mathematics\\ Indian Institute
of Science\\Bangalore-560 012}
\email{jyoti@math.iisc.ernet.in, sanjay@math.iisc.ernet.in, veluma@math.iisc.ernet.in}

\date{\today}
\thanks{ }

\begin{abstract}
We show that Riesz transforms associated to the Grushin operator
$ G = -\Delta-|x|^2 \partial_t^2 $ are bounded on $ L^p(\R^{n+1}).$ We also
establish an analogue of H\"{o}rmander-Mihlin multiplier theorem and study 
Bochner-Riesz means associated to the Grushin operator. The
main tools used are Littlewood-Paley theory and an operator valued
Fourier multiplier theorem due to L. Weis.

\end{abstract}

\maketitle

\section{Introduction}
\setcounter{equation}{0}

The aim of this paper is to study Riesz transforms and multipliers
associated to the Grushin operator $ G = -\Delta-|x|^2 \partial_t^2 $
on $ \R^n \times \R.$ The spectral decomposition of this non-negative
operator is explicitly known:
$$ Gf(x,t) = \frac{1}{2\pi}\int_{-\infty}^\infty  e^{-i\lambda t} \left(
\sum_{k=0}^\infty (2k+n)|\lambda| P_k(\lambda)f^\lambda(x)\right)
d\lambda $$
where $ P_k(\lambda)$ are the spectral projections of the scaled Hermite
operator $ H(\lambda) = -\Delta+\lambda^2 |x|^2 $ so that
$$  H(\lambda) = \sum_{k=0}^\infty (2k+n)|\lambda| P_k(\lambda) $$ and
$$ f^\lambda(x) = \int_{-\infty}^\infty f(x,t) e^{i\lambda t} dt.$$
Thus, given a bounded function $ m $ on $ \R $ we can define the multiplier
transformation $ m(G) $ by setting
$$ m(G)f(x,t)= \frac{1}{2\pi}\int_{-\infty}^\infty  e^{-i\lambda t} \left(
\sum_{k=0}^\infty m((2k+n)|\lambda|) P_k(\lambda)f^\lambda(x)\right)
d\lambda .$$ It is obvious that $ m(G) $ is bounded on $ L^2(\R^{n+1}) $ but
for the boundedness of $ m(G) $ on other $ L^p $ spaces we need to assume
more conditions on $ m.$

Recall that $ H(\lambda) $ can be written as
$$ H(\lambda) = \frac{1}{2} \sum_{j=1}^n (A_j(\lambda)A_j(\lambda)^*+
A_j(\lambda)^*A_j(\lambda)) $$
where
$$ A_j(\lambda) = -\frac{\partial}{\partial x_j}+\lambda x_j ,
 A_j(\lambda)^* = \frac{\partial}{\partial x_j}+\lambda x_j $$
are the creation and annihilation operators (for $ \lambda > 0$).
The operators
$$ R_j(\lambda) =  A_j(\lambda)H(\lambda)^{-\frac{1}{2}}, R_j^*(\lambda) =
 A_j(\lambda)^*H(\lambda)^{-\frac{1}{2}} $$
are Riesz transforms associated to the Hermite operator. It is therefore
natural to consider the operators
$$ R_jf(x,t)  = \frac{1}{2\pi}\int_{-\infty}^\infty  e^{-i\lambda t}
 R_j(\lambda)f^\lambda(x) d\lambda $$
and $ R_j^* $ similarly defined in terms of $ R_j^*(\lambda) $ and call them
the Riesz transforms for the Grushin operators. Again these Riesz transforms
are trivially bounded on $ L^2 $ but their $ L^p $ boundedness is far from trivial.
We prove:

\begin{thm} For $ 1 < p < \infty $ the Riesz transforms $ R_j, R_j^*,
j = 1,2,...,n $ are all bounded on $ L^p(\R^{n+1}).$
\end{thm}

We also consider higher order Riesz transforms which are defined as follows.
Let $\CH_{p,q} $ stand for the space of bigraded solid harmonics of bidegree
$ (p,q), p, q \in \N.$ Let $ G_\lambda(P) $ stand for the Weyl correspondence,
see Section 2.2 for the definition. Then it is known that $
G_\lambda(P)H(\lambda)^{-(p+q)/2} $ are bounded operators on $ L^p(\R^n), 1 <
p < \infty$\cite{pksstv}. We define higher order Riesz transforms for the Grushin operator
by setting
$$ R_Pf(x,t) = \frac{1}{2\pi}\int_{-\infty}^\infty e^{-i\lambda t}
G_\lambda(P)H(\lambda)^{-(p+q)/2}f^\lambda(x) d\lambda.$$ We prove:

\begin{thm} For any $ P \in \CH_{p,q} $ the Riesz transforms $ R_P $ are
bounded on $ L^p(\R^{n+1}), 1 < p < \infty.$
\end{thm}

Riesz transforms and higher order Riesz transforms associated to the
sublaplacian  on the Heisenberg group $ \H^n $ has been studied by several
authors, see e.g. \cite{CMZ},\cite{ERS} and \cite{ST}. Riesz transforms on
$ \H^n $ and more generally on nilpotent Lie groups turn out to be singular
integral operators. As we do not have a group structure behind the Grushin
operator it is not possible to use their techniques. We use a different
method described below. Concerning general multiplier transforms
$ m(G) $ we prove the following result.

\begin{thm} Let $ N \geq \frac{n}{2}+1 $ be an integer and let $ m \in
C^N(\R^*) $ satisfy the estimates
$ |m^{(k)}(\lambda)| \leq C_k |\lambda|^{-k} $ for $ |\lambda| $ large for
all $ k = 0,1,2,..,N.$ Then $ m(G) $ is bounded on $ L^p(\R^{n+1})$ for all
$ 1 < p < \infty.$
\end{thm}

In his thesis \cite{RM} R. Meyer has studied the wave equation associated to
the Grushin operator in one dimension. There he mentions about the
possibility of proving a multiplier theorem (stated as a conjecture)
for the Grushin operator. The above theorem gives such a result though the
proof is completely different from what he had in mind. We also have results
for the wave equation which will be presented in a forthcoming paper.

We can also treat Bochner-Riesz means associated to the Grushin operator. For
$ R > 0 $ Bochner-Riesz means $ B_R^\delta f $ of order $ \delta > 0 $ are
defined by
$$ B_R^\delta f(x,t) = \frac{1}{2\pi}\int_{-\infty}^\infty  e^{-i\lambda t}
\bigg(1-\frac{H(\lambda)}{R}\bigg)_+^\delta f^\lambda(x) d\lambda .$$ Here
$ \bigg(1-\frac{H(\lambda)}{R}\bigg)_+^\delta f^\lambda $ are the
Bochner-Riesz means associated to $ H(\lambda).$ For $ \lambda $ fixed these
means have been studied by various authors, see \cite{T2}. Concerning
$ B_R^\delta f $ for the Grushin operator we prove the following result.

\begin{thm} For $ \delta >(n+1)/2+1/6,$ the Bochner-Riesz means $ B_R^\delta $
are uniformly bounded on $ L^p(\R^{n+1}), 1 < p < \infty.$
\end{thm}

It may not be possible to improve the above result when $ n =1 $ as the
critical index for the Bochner-Riesz summability of one dimensional Hermite
expansions is $ 1/6 $, see \cite{T2}. However, the critical index for Hermite
expansions on $ \R^n, n \geq 2 $ is $ (n-1)/2 $ and hence it should be
possible to improve the above result. We conjecture that the above result is
true for $ \delta > (n+1)/2 $ for $ n\geq 2.$

We now briefly describe the methods used to prove these  theorems. If
$$ m(H(\lambda)) = \sum_{k=0}^\infty (2k+n)|\lambda| P_k(\lambda) $$ is the
multiplier transform for the Hermite operator then it follows that\
$$ (m(G)f)^\lambda(x) = m(H(\lambda))f^\lambda(x) .$$ Therefore, if we
identity $ L^p(\R^{n+1}) $ with $ L^p(\R,X) $ where $ X = L^p(\R^n) $ then
we can view $ m(G) $ as an operator valued Fourier multiplier (for the
Fourier transform on $ \R$) acting on $ L^p(\R,X) .$ Sufficient conditions
on $ m $ are known so that such multipliers are bounded on $L^p(\R,X) .$
Indeed, we make use of the following theorem of L. Weis \cite{LW}.

\begin{thm} Let  $ X $ and $ Y $ be UMD spaces. Let $ m:\R^* \rightarrow
B(X,Y) $ be a differentiable function such that the families
$ \{ m(\lambda): \lambda \in \R^* \} $ and $ \{ \lambda
\frac{d}{d\lambda}m(\lambda): \lambda \in \R^* \} $ are R-bounded. Then $ m $
defines a Fourier multiplier which is bounded from $ L^p(\R,X) $ into
$ L^p(\R,Y) $ for all $ 1 < p < \infty.$
\end{thm}

We need this theorem only for $ X = Y = L^p(\R^n) $ and in this case the
R-boundedness  is equivalent to a vector-valued inequality for $ m(\lambda) $
and $ \lambda \frac{d}{d\lambda}m(\lambda).$ Indeed, the R-boundedness of
a family of operators $ T(\lambda) $ is equivalent to the inequality
$$ \|\left(\sum_{j=1}^\infty |T(\lambda_j)f_j|^2 \right)^{\frac{1}{2}}\|_p
\leq C \|\left(\sum_{j=1}^\infty |f_j|^2 \right)^{\frac{1}{2}}\|_p $$
for all possible choices of $ \lambda_j \in \R^* $ and
$ f_j \in L^p(\R^n) .$ Thus we only need to verify this vector-valued
inequality for the two families in the theorem.

\section{Riesz transforms for the Grushin operator}
\setcounter{equation}{0}

\subsection{On the boundedness of  $ R_j $ and $ R_j^* $}

In this subsection we show that the Riesz transforms $ R_j $ and $ R_j^* $
defined in the introduction are all bounded on $ L^p(\R^{n+1}) $ as long as
$ 1 < p < \infty.$ This is done by showing that the operator valued functions
$ m_j(\lambda) = R_j(\lambda) $ and $ m_j^*(\lambda)= R_j^*(\lambda)$ satisfy
the conditions stated in the theorem of Weis. In view of the theorem of Weis
and the equivalent condition for R-boundedness Theorem 1.1 will follow once
we prove

\begin{thm} Let $ T(\lambda) $ be any of the families $ R_j(\lambda),
R_j^*(\lambda), \lambda \frac{d}{d\lambda}R_j(\lambda) $ or
$ \lambda \frac{d}{d\lambda}R_j^*(\lambda).$ Then the vector valued inequality
is satisfied:
$$ \|\left(\sum_{k=1}^\infty |T(\lambda_k)f_k|^2 \right)^{\frac{1}{2}}\|_p
\leq C \|\left(\sum_{k=1}^\infty |f_k|^2 \right)^{\frac{1}{2}}\|_p $$
for any $ 1 < p < \infty.$
\end{thm}

We only treat the cases of $ R_j(\lambda)$ and its derivative as the other
families are similarly dealt with. Without loss of generality let us assume
$ \lambda > 0.$ Recall that $ R_j(\lambda)= A_j(\lambda)
H(\lambda)^{-\frac{1}{2}}$ where $H(\lambda)^{-\frac{1}{2}} $ can be written
in terms of the Hermite semigroup $ e^{-tH(\lambda)} $ as
$$ H(\lambda)^{-\frac{1}{2}} = \int_0^\infty t^{-\frac{1}{2}}
e^{-tH(\lambda)} dt.$$
An orthonormal basis for $ L^2(\R^n) $ consisting of eigenfunctions of
$ H(\lambda) $ are provided by $ \Phi_\alpha^\lambda(x) = \lambda^{n/4}
\Phi(\lambda^{\frac{1}{2}}x) $ where $ \Phi_\alpha $ are the Hermite
functions on $ \R^n $ satisfying $ H\Phi_\alpha = (2|\alpha|+n)\Phi_\alpha.$
Here $ H=H(1) $ is the Hermite operator. Thus $ \Phi_\alpha^\lambda $
are eigenfunctions of $ H(\lambda)$ with eigenvalues $ (2|\alpha|+n)\lambda.$
The spectral projections $ P_k(\lambda) $ of $ H(\lambda) $ are defined by
$$ P_k(\lambda)f =  \sum_{|\alpha| =k} (f,\Phi_\alpha^\lambda)
\Phi_\alpha^\lambda.$$

The kernel $ h_t^\lambda(x,y) $ of the semigroup $ e^{-tH(\lambda)} $ is
explicitly known, thanks to Mehler's formula, and hence the kernel of
$ R_j(\lambda) $ is given by $ A_j(\lambda)h_t^\lambda(x,y).$  It turns out
that $ R_j(\lambda) $ are all Calderon-Zygmund singular integrals. Indeed,
Stempak and Torrea \cite{ST} have shown that
$$ | A_j(1)h_t^1(x,y)| \leq C |x-y|^{-n} $$
 and
$$ | \frac{\partial}{\partial x_i} A_j(1)h_t^1(x,y)|+
|\frac{\partial}{\partial y_i} A_j(1)h_t^1(x,y)| \leq C_i |x-y|^{-n-1}$$
(see Theorem 3.3 in \cite{ST}). As can be easily verified $  A_j(\lambda)h_t^\lambda(x,y) =
\lambda^{\frac{n}{2}} (A_j(1)h_t^1)(\lambda^{1/2}x,\lambda^{1/2}y)$ and hence
the kernels of $ R_j(\lambda) $ are Calderon-Zygmund  kernels and CZ constants
are independent of $ \lambda.$ Now we can appeal to the vector valued
inequalities for CZ singular integrals proved by Cordoba and Fefferman (see
Theorem 1.3 Chapter XII in \cite{AT}). Thus the Riesz transforms
$ R_j(\lambda)$ are R-bounded.

We now turn our attention to the derivative of $ R_j(\lambda).$ Denoting the
kernel of $ R_j(1) $ by $ R_j(x,y) $ and noting that the kernel of
$ R_j(\lambda) $ is $ \lambda^{n/2}R_j(\lambda^{1/2}x,\lambda^{1/2}y) $ we
can easily prove the following.

\begin{lem} $ \lambda \frac{d}{d\lambda}R_j(\lambda) $ is a linear
combination of $ R_j(\lambda) $ and an operator $ T_j(\lambda) $ given by
$$ T_j(\lambda)f(x) = \lambda^{n/2} \int_{\R^n}
\left((x\cdot \nabla_x+ y\cdot \nabla_y )R_j\right)(\lambda^{1/2}x,\lambda^{1/2}y) f(y) dy $$
where $ x\cdot \nabla_x = \sum_{i=1}^n x_i \frac{\partial}{\partial x_i}$ and
 $ y\cdot \nabla_y = \sum_{i=1}^n y_i \frac{\partial}{\partial y_i}.$
\end{lem}

Let $ \delta_\lambda $ stand for the dilation operator $ f(x) \rightarrow
f(\lambda^{1/2}x).$ Then we have $ T_j(\lambda) =
\delta_\lambda T_j(1)\delta_\lambda^{-1}.$ Integrating by parts in
$$ T_j(1)f(x) = \int_{\R^n}  \left(\sum_{i=1}^n(x_i
\frac{\partial}{\partial x_i}+y_i\frac{\partial}{\partial y_i})R_j\right)
(x,y) f(y) dy $$ we see that it is a sum of $ -R_j(1) $ and the commutator of
$ x\cdot \nabla_x $ with $ R_j(1).$ Since
$ 2x_i = A_i(1)+A_i^*(1) $ and $ 2\frac{\partial}{\partial x_i} =
A_i^*(1)-A_i(1) $ we see that
$$  x\cdot \nabla_x = \frac{1}{4}\sum_{i=1}^n
(A_i^{*2}-A_i^2+[A_i,A_i^*]) = n/2 I + \frac{1}{4}\sum_{i=1}^n(A_i^{*2}-A_i^2) $$ as $ [A_i,A_i^*] =2I.$ Thus  the commutator of
$  x\cdot \nabla_x $ with $ R_j(1)$ reduces to a
sum of commutators of the form $ [A_i^{*2},R_j(1)] $ and $ [A_i^2,R_j(1)]$ and we are left with proving

\begin{prop} The families $ \delta_\lambda [A_i^{*2},R_j(1)]\delta_\lambda^{-1}$ and $ \delta_\lambda [A_i^{2},R_j(1)]\delta_\lambda^{-1}$ are R-bounded.
\end{prop}
\begin{proof} We only consider the family
$ \delta_\lambda [A_i^{2},R_j(1)]\delta_\lambda^{-1}$ as the treatment of the
other one is similar. We will also assume $ i = j$ as the other cases are
simpler as $ A_i $ commutes with $ A_j $ when $ i $ is different from $ j.$
A moment's thought reveals that
$$ [A_j^{2},R_j(1)] = A_j[A_j, R_j(1)]+[A_j, R_j(1)]A_j $$ and hence we are
left with the families
$$ \delta_\lambda A_j[A_j, R_j(1)] \delta_\lambda^{-1},~~~~
\delta_\lambda [A_j, R_j(1)] A_j \delta_\lambda^{-1} $$
and again we will consider only the first family.

As
$$  A_j[A_j, R_j(1)] = A_j^3H^{-1/2}-A_j^2H^{-1/2}A_j = A_j^2[A_j,H^{-1/2}] $$
making use of the fact that $ H^{-1/2}A_j = A_j(H+2)^{-1/2} $ we get
$$  A_j[A_j, R_j(1)] = -A_j^3 ((H+2)^{-1/2}-H^{-1/2}).$$  An easy calculation
shows that
$$  \delta_\lambda A_j \delta_\lambda^{-1} = \lambda^{-1/2} A_j(\lambda),~~~~
 \delta_\lambda H^{-1/2} \delta_\lambda^{-1} = \lambda^{1/2}
H(\lambda)^{-1/2} $$ and hence we finally get
$$  \delta_\lambda A_j[A_j, R_j(1)] \delta_\lambda^{-1} = -\lambda^{-1}
A_j(\lambda)^3 ((H(\lambda)+2\lambda)^{-1/2}-H(\lambda)^{-1/2}).$$
The above can be rewritten as
$$ A_j(\lambda)^3 \int_0^1 (H(\lambda)+2\lambda s)^{-3/2} ds .$$
Thus everything boils down to showing that the family $ A_j(\lambda)^3
(H(\lambda)+2\lambda s)^{-3/2} $ is R-bounded uniformly in $ s, 0 < s < 1.$
But these operators are also CZ singular integrals whose kernels satisfy
estimates uniformly in $ \lambda $ and $ s$. This completes the proof.
\end{proof}

\subsection{Higher order Riesz transforms}

We briefly recall some notations we use referring to \cite{T2} for details.
Given a function $ f $ on $ \C^n $ we denote by $ W_\lambda(f) $ the Weyl
transform defined by
$$ W_\lambda(f) = \int_{\C^n} f(z) \pi_\lambda(z,0) dz $$
where $ \pi_\lambda $ stands for the Schr\"{o}dinger representation of the
Heisenberg group realised on $ L^2(\R^n).$ For $ P \in \CH_{p,q} $ the
operator $ G_\lambda(P) = W_\lambda(\CF_\lambda P) $ is called its Weyl
correspondence. Here $ \CF_\lambda $ stands for the symplectic Fourier
transform. The unitary group $ U(n) $ acts on $ \CH_{p,q} $ by $
\rho(\sigma)P(z) = P(\sigma^{-1}z) $ and the action is irreducible, i.e. we
get an irreducible unitary representation of $ U(n) $ on $ \CH_{p,q} .$
In view of this, if $ P \in  \CH_{p,q} $ then the linear span of
its orbit $ \rho(\sigma)P $ under $ U(n) $ is the whole of $ \CH_{p,q}.$
Therefore, if we let $ P_0(z) = z_1^p\bar{z_2}^q $ then any $ P \in
\CH_{p,q} $ is a linear combination of $ \rho(\sigma)P_0, \sigma \in U(n).$

\begin{prop} If the higher order Riesz transform $ R_{P_0} $ is bounded
on $ L^p(\R^{n+1}) $ then so is $ R_P $ for any $ P \in \CH_{p,q}.$
\end{prop}
\begin{proof} In order to prove this proposition we need to use several facts
about the symplectic group $ Sp(n,\R) $ and the metaplectic representations
$ \mu_\lambda(\sigma) $ of $ Sp(n,\R) .$ A good source for the material we
use here is Chapter 4 of Folland \cite{F}. The action of $ U(n) $ on the
Heisenberg group leads to certain unitary operators, denoted by
$ \mu_\lambda(\sigma) $ so that
$$  \mu_\lambda(\sigma) \pi_\lambda(z,t) \mu_\lambda(\sigma)^*
= \pi_\lambda(\sigma z,t) .$$
Recalling the definition of $ R_P $ we see that this means
$$ R_{\rho(\sigma)P}f(x,t) = \frac{1}{2\pi} \int_{-\infty}^\infty
e^{-i\lambda t}\mu_\lambda(\sigma)
G_\lambda(P) \mu_\lambda(\sigma)^*f^\lambda(x) d\lambda.$$
Therefore, our proposition will be proved if we show that the operators
$$ f \rightarrow \frac{1}{2\pi} \int_{-\infty}^\infty
e^{-i\lambda t}\mu_\lambda(\sigma)f^\lambda(x) d\lambda $$
are all bounded on $ L^p(\R^{n+1}).$

If we identify
$ \C^n $ with $ \R^{2n} $ the group $ U(n) $ corresponds to $ Sp(n,\R)\cap
O(2n,\R).$ By Proposition 4.10 in \cite{F} the group  $ Sp(n,\R) $ is
generated by certain subgroups $   N = \left\{  \left( \begin{array}{cc}I & 0  \\C & I  \end{array} \right): C= C^* \right\}$   and \\ $ D= \left\{  \left( \begin{array}{cc}
A & 0  \\
0 & A^{*-1}  \end{array} \right): A \in GL(n,\R) \right\} $ and the matrix $ J =  \left( \begin{array}{cc}
0 & I  \\
-I & 0  \end{array} \right).$ For
elements of these subgroups $ \mu_\lambda(\sigma) $ are explicitly known. In
fact we have:
$$ \mu_\lambda(\sigma)\varphi(\xi) = |A|^{-1/2}\varphi(A^{-1}\xi) $$ if
$ \sigma \in D $ and
$$ \mu_\lambda(\sigma)\varphi(\xi) = e^{\frac{i}{2}\lambda \xi \cdot C\xi}
\varphi(\xi) $$ if  $ \sigma \in  N.$ Moreover,
$$  \mu_\lambda(J)\varphi(\xi) = \hat{\varphi}(\lambda \xi) $$ is the Fourier
transform on $ \R^n $ followed by a dilation. Clearly, the operator
$$  f \rightarrow \frac{1}{2\pi} \int_{-\infty}^\infty
e^{-i\lambda t} |A|^{-1/2}f^\lambda(A^{-1}x) d\lambda  =
|A|^{-1/2}f(A^{-1}x,t) $$
is bounded on $ L^p(\R^{n+1}).$  And so is the operator
$$  f \rightarrow \frac{1}{2\pi} \int_{-\infty}^\infty
e^{-i\lambda t} e^{\frac{i}{2}\lambda \xi \cdot C\xi}f^\lambda(x)
d\lambda  = f(x, t-\frac{1}{2}x\cdot Cx) .$$

Though Fourier transform does not define a bounded operator on $ L^p $
unless $ p =2 $ we can take care of the operators involving $ \mu_\lambda(J)$
in the following way. Observe that we need to prove the boundedness of
$$  f \rightarrow \frac{1}{2\pi} \int_{-\infty}^\infty
e^{-i\lambda t} \mu_\lambda(J)G_\lambda(P)\mu_\lambda(J)^*f^\lambda(x)
d\lambda $$
whenever the operator
$$  f \rightarrow \frac{1}{2\pi} \int_{-\infty}^\infty e^{-i\lambda t}
G_\lambda(P)f^\lambda(x) d\lambda $$
is bounded. But as $ P \in \CH_{p,q} $ is homogeneous of bidegree $ (p,q) $
and the action of $ J $ is $ \rho(J)P(z) = P(-iz) = (-i)^{p-q}P(z) $ the
former operator is just a scalar multiple of the latter and hence bounded.
As $ Sp(n,\R) $ is generated by  $ N , D $ and $ J $ the proposition is
proved.
\end{proof}

In view of the above proposition Theorem 1.2 will be proved once we obtain the
following result.

\begin{thm} Let $ P_0(z) = z_1^p\bar{z_2}^q.$ Then $ R_{P_0}$ is bounded
on $ L^p(\R^{n+1}) $ for $ 1 < p < \infty.$
\end{thm}

\begin{proof}
When $ P_0(z) = z_1^p \bar{z}_2^q $,  $G_\lambda(P_0)$ given by $ A_2(\lambda)^q A_1(\lambda)^{*p},$ see
Geller \cite{G} or \cite{T2}.   Hence the multiplier corresponding to $ R_{P_0}$ is given by $A_2(\lambda)^{q} A_1(\lambda)^{*p} H(\lambda)^{-\frac{p+q}{2}} $. We prove the boundedness of  $ R_{P_0}$ by showing that this multiplier can be written as the product of the multipliers corresponding to first order Riesz transform and their adjoints.  Since $ A_2(\lambda)A_1(\lambda)^* $ commutes with $ H(\lambda) $, we can write $A_2(\lambda)^{q} A_1(\lambda)^{*p} H(\lambda)^{-\frac{p+q}{2}} $ as the product of operators of the form
$H(\lambda)^{-1/2}A_2(\lambda)$, $A_1(\lambda)^*H(\lambda)^{-1/2}$ and $\lambda H(\lambda)^{-1}$.  We refer to \cite{pksstv} for details.  $\lambda H(\lambda)^{-1}$  can be written as$$ \frac{1}{2}H(\lambda)^{-\frac{1}{2}}\left[A_1^*(\lambda)A_1(\lambda)  - A_1(\lambda)A_2(\lambda)^* \right]H(\lambda)^{-\frac{1}{2}}. $$ $ A_1(\lambda)^*H(\lambda)^{-1/2}$ is the multiplier corresponding to the first order Riesz transform $ R_1^*(\lambda)$ and $H(\lambda)^{-1/2}
A_2(\lambda)$ is the multiplier corresponding to the adjoint of $ R_2^*(\lambda).$ The proof is now complete as we have already shown that the operators corresponding to multipliers of these forms are bounded.
\end{proof}
\section{ On the boundedness of multiplier transforms}
\setcounter{equation}{0}

In this section we prove Theorem 1.3. In view of Theorem 1.5 it is enough to
show that the families $ m(H(\lambda))$ and
$ \lambda \frac{d}{d\lambda}m(H(\lambda))$ are R-bounded. We remark that we
only need to prove vector valued inequalities for these families.

\subsection{R-boundedness of the family $ m(H(\lambda))$}

For a fixed $ \lambda $ conditions on $ m $ are known so that $ m(H(\lambda))$
is bounded on $ L^p(\R^n), 1< p < \infty $, see e.g. Theorem 4.2.1
in \cite{T2}. Under the same assumptions on $ m $ it is possible to prove
a vector valued inequality for the sequence of operators $ m(H(\lambda_j))$
for any choice of $ \lambda_j.$

\begin{thm} Assume that $ m \in C^k(\R^*) $ satisfies $ |m^{(j)}(\lambda)|
\leq C_j |\lambda|^{-j} $ for large $|\lambda|$ and for all $ j=0,1,2,..,k $
where $ k \geq n/2.$ Then for any choice of $ \lambda_j \in \R^* $ and $ f_j
\in L^p(\R^n), 1 < p < \infty $ we have the uniform estimate
$$ \| (\sum_{j=0}^\infty | m(H(\lambda_j))f_j|^2)^{1/2}\|_p \leq C
\|(\sum_{j=0}^\infty |f_j|^2)^{1/2}\|_p.$$
\end{thm}

In proving this theorem we closely follow the proof of Theorem 4.2.1 in
\cite{T2} and hence we briefly recall the proof. Fixing $ \lambda > 0 $ for
the  sake of definiteness we consider the boundedness of $ m(H(\lambda)) $ on
$ L^p(\R^n).$ In \cite{T2} this is achieved by means of Littlewood-Paley
g-functions. We define, for each positive integer $ k $, the function
$$ (g_k^\lambda(f,x))^2 = \int_0^\infty  |\partial_t^kT_t^\lambda f(x)|^2
t^{2k-1} dt $$
where $ T_t^\lambda = e^{-tH(\lambda)} $ is the semigroup generated by
$ H(\lambda).$ In \cite{T2} it is shown that (see Theorem 4.1.2)
$ g_1^\lambda(f) $ can be
considered as a singular integral operator whose kernel is taking values
in the Hilbert space $ L^2(\R^+,tdt) $ and hence bounded on $ L^p(\R^n).$
Moreover, it is also shown that $ g_k^\lambda(f) $ is a constant multiple of
isometry on $ L^2(\R^n) $ and hence we have the equivalence of norms:
$$ C_1 \|f\|_p \leq \|g_1^\lambda(f)\|_p \leq C_2 \|f\|_p $$
for all $ 1 < p < \infty.$ Here, it can be verified that the constants
$ C_j, j=1,2 $ can be taken independent of $ \lambda.$ Moreover, we also have
the pointiest estimate $ g_k^\lambda(f,x) \leq C_k g_{k+1}^\lambda(f,x).$

A version of Cordoba-Fefferman theorem for CZ singular integrals whose
kernels are taking values in a Hilbert space is true. In fact an easy
modification of the proof given in Torchinsky \cite{AT} substantiates this
claim. By treating the g-functions as singular integral operators with kernels
taking values in $ L^2(\R^+,tdt) $ we can deduce the following result.

\begin{thm} For any choice of $ \lambda_j \in \R^* $ and $ f_j \in L^p(\R^n)$
we have
$$ \| (\sum_{j=0}^\infty | g_1^{\lambda_j}(f_j)|^2)^{1/2}\|_p \leq C
\|(\sum_{j=0}^\infty |f_j|^2)^{1/2}\|_p $$
for all $ 1 < p < \infty.$
The reverse inequality also holds.
\end{thm}

In order to prove the multiplier theorem we also need to consider
$ g_k^{*\lambda} $ functions which are defined by
$$ (g_k^{*\lambda}(f,x))^2 = \int_{\R^n}\int_0^\infty t^{-n/2}
(1+t^{-1}|x-y|^2)^{-k}|\partial_tT_t^\lambda f(y)|^2 t dt dy.$$
If $ k > n/2 $ it is proved in \cite{T2} (see Theorem 4.1.3) that
$$ \|g_k^{*\lambda}(f)\|_p \leq C \|f\|_p,~~~~~~ p > 2.$$ The proof of this
depends on two facts: the boundedness of $ g_1^\lambda $ functions and the
Hardy-Littlewood maximal functions. By the celebrated theorem of
Fefferman-Stein the Hardy-Littlewood maximal function satisfies a vector
valued inequality, see Theorem 1.1, Chapter XII in \cite{AT}. Hence by
combining this with Theorem 3.2 above we get

\begin{thm} For any choice of $ \lambda_j \in \R^* $ and $ f_j \in L^p(\R^n)$
we have                                                                        $$ \| (\sum_{j=0}^\infty | g_k^{*\lambda_j}(f_j)|^2)^{1/2}\|_p \leq C
\|(\sum_{j=0}^\infty |f_j|^2)^{1/2}\|_p $$  for all $ p > 2 $
provided $ k > n/2.$
\end{thm}

Under the hypothesis on $ m $ it has been proved in \cite{T2}
(see Section 4.2) that the pointwise inequality
$$ g_{k+1}^\lambda(m(H(\lambda))f, x) \leq C_k g_k^{*\lambda}(f,x) $$ holds.
It can be checked, by following the proof carefully, that the constant
$ C_k $ can be taken independent of $ \lambda.$ It is then clear that a
vector valued analogue of the above estimate is valid. Hence, by appealing
to Theorems 3.2 and 3.3 we obtain Theorem 3.1. This takes care of the
R-boundedness of $ m(H(\lambda)).$

\subsection{R-boundedness of the derivative of $ m(H(\lambda))$}

We begin with the observation that the kernel of $  m(H(\lambda)) $ is
given by
$$ \lambda^{n/2} \sum_{k=0}^\infty m((2k+n)\lambda)
\Phi_k(\lambda^{1/2}x,\lambda^{1/2}y) $$
where $ \Phi_k(x,y) $ is the kernel of $ P_k.$ If we let $ m_\lambda(t) =
m(t\lambda) $ then, with obvious notations, we have
$$ m(H(\lambda))(x,y) = \lambda^{n/2}
m_\lambda(H)(\lambda^{1/2}x,\lambda^{1/2}y).$$ Using this we can easily prove

\begin{lem} $\lambda \frac{d}{d\lambda}m(H(\lambda)) $ is a linear
combination of operators of the form $ m(H(\lambda)),
H(\lambda)m'(H(\lambda)) $ and the commutators $ \delta_\lambda 
[A_j^2,m_\lambda(H)]\delta_\lambda^{-1} $ and $ \delta_\lambda [A_j^{*2},
m_\lambda(H)]\delta_\lambda^{-1} .$
\end{lem}

We have just shown that the families $ m(H(\lambda))$ and
$ H(\lambda)m'(H(\lambda)) $ are R-bounded under the hypothesis of Theorem 1.3.
To handle the commutators we proceed as in the case of Riesz transforms. We
need to consider the operators
$ \delta_\lambda A_j[A_j,m_\lambda(H)] \delta_\lambda^{-1} $ and
$ \delta_\lambda [A_j,m_\lambda(H)]A_j \delta_\lambda^{-1} .$
As before using $ m_\lambda(H)A_j = A_j m_\lambda(H+2) $ we get
$$ \delta_\lambda A_j[A_j,m_\lambda(H)] \delta_\lambda^{-1}= -\delta_\lambda
A_j^2 (m_\lambda(H+2)-m_\lambda(H)) \delta_\lambda^{-1}.$$
Finally, the above can be written in the form
$$ A_j(\lambda)^2 \int_0^1 m'(H(\lambda)+2\lambda s)) ds.$$
Since $ A_j(\lambda)^2H(\lambda)^{-1} $ are singular integral operators we
need to check the R-boundedness of
$$ \int_0^1 H(\lambda) m'(H(\lambda)+2\lambda s)) ds.$$
The hypothesis on $ m $ allows us to handle this family with estimates uniformly in $ s, 0< s < 1.$ This completes the proof of the R-boundedness of
$\lambda \frac{d}{d\lambda}M(H(\lambda)) $ and hence Theorem 1.3 is proved.

\subsection{Boundedness of the Bochner-Riesz means}

Recall that  the kernel of $ P_k(\lambda) $ is of the form  $ |\lambda|^{n/2}
\Phi_k(|\lambda|^{1/2}x,|\lambda|^{1/2}y) $ and hence
$$ B_R^\delta f(x,t) = \frac{1}{2\pi} \int_{-\infty}^\infty \int_{\R^n}
e^{-i\lambda t}|\lambda|^{n/2}
s_{R/|\lambda|}^\delta(|\lambda|^{1/2}x,|\lambda|^{1/2}y)
f^\lambda(y)dy d\lambda.$$
where
$$ s_{R}^\delta(x,y) = \sum_{(2k+n)\leq R}(1-(2k+n)/R)_+^\delta \Phi_k(x,y).$$
If $ D_rf(x,t) = r^{n+2}f(rx,r^2t) $ stands for the nonisotropic
dilation, then it is easy to see  that
$$ B_R^\delta D_r f(x,t) = D_r B_{R/r^2}^\delta f(x,t) $$ and hence
it is enough to prove the boundedness of $ B_1^\delta.$ Note that
$$ B_1^\delta f(x,t) =  \frac{1}{2\pi} \int_{-\infty}^\infty e^{-i\lambda t}
m_\delta(\lambda)f^\lambda(x) d\lambda $$
where
$$ m_\delta(\lambda) = \sum_{(2k+n)|\lambda|\leq 1}(1-(2k+n)|\lambda|)^\delta
P_k(\lambda) .$$
Thus we can view $ B_1^\delta $ as a Fourier multiplier corresponding to
the operator valued multiplier $ m_\delta(\lambda).$ Therefore, in order to prove Theorem 1.4 we only need to show that $ m_\delta(\lambda) $ and $ \lambda
\frac{d}{d\lambda}m_\delta(\lambda) $ are both R-bounded families of
operators.

Let
$$ S_R^\delta f(x) =  \sum_{(2k+n)\leq R}(1-(2k+n)/R)_+^\delta P_kf $$
stand for the Bochner-Riesz means associated to the Hermite operator $ H.$
Then it follows that
$$ m_\delta(\lambda) = \delta_{|\lambda|}
S_{|\lambda|^{-1}}^\delta \delta_{|\lambda|}^{-1} .$$
We make use of the following result in order to prove the R-boundedness of the
family $ m_\delta(\lambda).$ Let $ Mf $ stand for the Hardy-Littlewood maximal
function of $ f.$

\begin{thm} Assume that $ f \in L^p(\R^n), 1 \leq p \leq \infty $ and
$ \delta > (n-1)/2+1/6.$
Then $ \sup_{R > 0}|S_R^\delta f(x)| \leq C \bigg(Mf(x)+Mf(-x)\bigg).$
\end{thm}

This theorem can be proved using the estimates given in \cite{T2}
(see Theorem 3.3.5). Details can be found in \cite{T1}. Using
this result it is not difficult to prove

\begin{thm} The family $ m_\delta(\lambda) $ is R-bounded on $ L^p(\R^n) $
for all $ \delta > (n-1)/2+1/6 $ and $ 1 < p < \infty.$
\end{thm}
\begin{proof}  In view of the relation between $ m_\delta(\lambda) $ and
$ S_R^\delta $ we need to show that
$$ \left \| \left(\sum_{j=1}^\infty
|\delta_{R_j}^{-1}S_{R_j}^\delta \delta_{R_j}f_j|^2 \right)^{1/2} \right \|_p  \leq C \left \| \left(\sum_{j=1}^\infty |f_j|^2 \right)^{1/2} \right \|_p $$
for all sequences $ R_j > 0 $ and $ f_j \in L^p(\R^n).$ By the result of
Theorem 3.5
$$ |\delta_{R_j}^{-1}S_{R_j}^\delta \delta_{R_j}f_j(x)|
\leq C \bigg(M(\delta_{R_j}f_j)(R_j^{-1/2}x)+
M(\delta_{R_j}f_j)(-R_j^{-1/2}x)\bigg) $$ and a simple calculation,
recalling the definition of $ Mf $, shows that
$$ |\delta_{R_j}^{-1}S_{R_j}^\delta \delta_{R_j}f_j(x)| \leq C Mf_j(x).$$
Thus we are required to prove
$$  \left \| \bigg(\sum_{j=1}^\infty
|Mf_j|^2 \bigg)^{1/2} \right \|_p  \leq C \left \| \bigg(\sum_{j=1}^\infty |f_j|^2 \bigg)^{1/2} \right \|_p $$
which is the Fefferman-Stein inequality for the maximal function.
\end{proof}

It remains to be proved that the family
$ \lambda \frac{d}{d\lambda}m_\delta(\lambda) $ is also R-bounded. This is done in the following theorem.

\begin{thm} For $ \delta > (n+1)/2+1/6 $ the family
$ \lambda \frac{d}{d\lambda}m_\delta(\lambda) $ is R-bounded on
$ L^p(\R^{n+1}), 1 < p < \infty.$
\end{thm}
\begin{proof} Applying Lemma 3.4 with $ m(t) = (1-t)_+^\delta $ we see that
$ \lambda \frac{d}{d\lambda}m_\delta(\lambda) $ is a linear combination of
$ (1-H(\lambda))_+^\delta,
H(\lambda)((1-H(\lambda))_+^{\delta-1} $ and commutators of the form $
[A_j(\lambda)^2, (1-H(\lambda))_+^\delta] $ and
$ [A_j(\lambda)^{*2}, (1-H(\lambda))_+^\delta] .$ The first two families are handled as in the previous theorem. We prove the R-boundedness of
$ [A_j(\lambda)^2, (1-H(\lambda))_+^\delta] $ and the other family can be
similarly handled. It is enough to look at $
A_j(\lambda)^2 (1-H(\lambda))_+^\delta  $  which we write as
$$ A_j(\lambda)^2 (1-H(\lambda))_+^\delta  = A_j(\lambda)^2 H(\lambda)^{-1}
H(\lambda)(1-H(\lambda))_+^\delta  .$$
But $ A_j(\lambda)^2 H(\lambda)^{-1}$ are singular integral operators and
$$ H(\lambda)(1-H(\lambda))_+^\delta = (1-H(\lambda))_+^\delta -
(1-H(\lambda))_+^{\delta+1} .$$ Thus
$ A_j(\lambda)^2 (1-H(\lambda))_+^\delta  $ is R-bounded.
\end{proof}

\begin{center}
{\bf Acknowledgments}

\end{center}
The work of the last author is supported  by J. C. Bose Fellowship from
the Department of Science and Technology (DST).

\end{document}